\documentclass[11pt,twoside]{article}
\usepackage{authblk}
\usepackage[scaled=0.8]{DejaVuSansMono}
\usepackage{emptypage}
\usepackage[a4paper,marginratio={5:5,5:7},width=147mm]{geometry}
\usepackage[T2A,T1]{fontenc}
\usepackage[utf8x]{inputenc}

\usepackage{fancyhdr}
\newcommand{\maybebf}{}

\fancypagestyle{plain}{%
    \fancyhf{}%
    \fancyfoot[C]{\maybebf \thepage}
    
}
\fancypagestyle{intro}{
    \renewcommand{\sectionmark}[1]{\markboth{##1}{}}
    \fancyhf{}
    \fancyhead[LE,RO]{\leftmark}
    \fancyfoot[C]{\maybebf \thepage}

}

\renewcommand{\sectionmark}[1]{%
  \markboth{%
    \ifnum\value{section}>0
      \maybebf{\thesection.}\space
    \fi
    #1%
  }{%
    \ifnum\value{subsection}=0
            \thesection. #1
        \fi%
  }%
}
\renewcommand{\subsectionmark}[1]{%
    \markright{%
        \ifnum\value{subsection}>0
            \thesubsection. #1
        \fi%
    }%
}
\usepackage[english]{babel}
\usepackage{graphicx}
\usepackage{amsmath}
\usepackage{amssymb}
\usepackage{amsthm}
\usepackage{mathtools}
\usepackage[final]{microtype}
\usepackage{tikz}
\usetikzlibrary{arrows.meta,arrows,decorations.pathmorphing,shapes,knots,hobby,quotes,positioning}
\usepackage{asymptote}
\usepackage{tikz-cd}
\usepackage{parskip}
\usepackage{float}
\usepackage{caption}
\usepackage{subcaption}
\usepackage{enumitem}
\usepackage{adjustbox}
\usepackage{xfakebold}
\usepackage{mdframed}
\usepackage{booktabs}
\usepackage{multirow}
\usepackage{siunitx}
\usepackage{nicematrix}
\NiceMatrixOptions{xdots/shorten=0.87em, renew-dots}
\usepackage{diagbox}
\usepackage[calc,showdow,english]{datetime2}
\DTMnewdatestyle{mydateformat}{%
}

\usepackage{comment}
\usepackage{import}
\usepackage{standalone}
\usepackage{xcolor}

\definecolor{zg}{gray}{0.6}

\usepackage{csquotes}
\usepackage[doi=false,isbn=false,
maxbibnames=99,backend=biber,style=numeric,sorting=nyt]{biblatex}
\usepackage[unicode]{hyperref}

\definecolor{hlinkcol}{HTML}{A00000}
\definecolor{hcitecol}{HTML}{308030}
\hypersetup{
    colorlinks,
    linkcolor={teal!100!black},
    citecolor={green!50!black},
    urlcolor={blue!0!black}
}

\theoremstyle{plain}
\newtheorem{thm}{Theorem}[section]

\newcommand{\thistheoremname}{}
\newtheorem{genericthm}[thm]{\thistheoremname}

\newtheorem*{genericthm*}{\thistheoremname}
\newenvironment{namedthm*}[1]
  {\renewcommand{\thistheoremname}{#1}%
   \begin{genericthm*}}
  {\end{genericthm*}}

\newtheorem{prop}[thm]{Proposition}
\newtheorem{lemma}[thm]{Lemma}
\newtheorem{cor}[thm]{Corollary}
\newtheorem{conj}[thm]{Conjecture}
\newtheorem*{thm*}{Theorem}
\newtheorem*{thm20*}{Theorem \ref*{thm:main20}}
\newtheorem*{lemma*}{Lemma}
\newtheorem*{prop*}{Proposition}
\newtheorem*{conj*}{Conjecture}
\newtheorem*{cor*}{Corollary}

\theoremstyle{definition}
\newtheorem{defin}[thm]{Definition}

\newtheorem*{defin*}{Definition}

\theoremstyle{remark}
\newtheorem{remark}[thm]{Remark}

\counterwithin{figure}{section}
\counterwithin{table}{section}

\newcommand{\surj}{\twoheadrightarrow}

\newcommand{\id}{\mathrm{id}}
\newcommand{\ZZ}{\mathbb Z}

\newcommand{\RR}{\mathbb R}

\newcommand{\til}{\widetilde}
\newcommand{\blt}{*}
\newcommand{\poly}[1]{\overline{#1}}
\tikzset{
dot/.style = {circle, fill, minimum size=#1,
              inner sep=0pt, outer sep=0pt},
dot/.default = 5pt
}

\DeclarePairedDelimiter\abs{\lvert}{\rvert}%
\DeclarePairedDelimiter\norm{\lVert}{\rVert}%

\makeatletter
\let\oldabs\abs
\def\abs{\@ifstar{\oldabs}{\oldabs*}}
\let\oldnorm\norm
\def\norm{\@ifstar{\oldnorm}{\oldnorm*}}
\makeatother

\DeclareMathOperator{\image}{im}
\DeclareMathOperator{\coker}{coker}
\DeclareMathOperator{\Hom}{Hom}
\DeclareMathOperator{\ab}{ab}
\DeclareMathOperator{\rk}{rk}

\DeclareMathOperator{\vspan}{span}

\DeclareMathOperator{\gcrd}{gcrd}

\newcommand{\fbseries}{
\unskip\setBold\aftergroup\unsetBold\aftergroup\ignorespaces}
\makeatletter
\newcommand{\setBoldness}[1]{\def\fake@bold{#1}}
\makeatother

\tikzset{
    rots/.style={anchor=south, rotate=90, inner sep=.5mm}
}

\newcommand\restr[2]{{%
  \left.\kern-\nulldelimiterspace %
  #1 %
  \vphantom{\big|} %
  \right|_{#2} %
  }}

\author{Jacopo G. Chen\thanks{Shanghai Institute for Mathematics and Interdisciplinary Sciences (SIMIS), Shanghai, 200433, China. Email: \href{mailto:jacopo.chen@simis.cn}{\texttt{jacopo.chen@simis.cn}}}}
\title{On a twisted Singer conjecture}

\begin{document}

\DTMsetdatestyle{mydateformat}
\date{}
\maketitle

\begingroup
\centering\small
\textbf{Abstract}\par\smallskip
\begin{minipage}{\dimexpr\paperwidth-9.7cm}
According to the Singer conjecture, the $L^2$-Betti numbers of a closed aspherical manifold are expected to vanish outside the middle dimension. In this paper, we study a natural analog of the Singer conjecture concerning the vanishing of \emph{twisted} $L^2$-Betti numbers outside the lower middle dimension. These invariants can be thought of as $L^2$-Betti numbers of an infinite cyclic cover, depending on a first cohomology class. This is related to a previous conjecture by the author, which connects the twisted $L^2$-Euler characteristic to the Thurston norm, and which we prove for a class of closed arithmetic hyperbolic manifolds.
\end{minipage}
\par\endgroup

\section{Introduction}

The Singer conjecture states that, if M is a closed aspherical manifold, then the $L^2$-Betti numbers $b_k^{(2)}(\widetilde M)$ vanish for $2k \ne \dim M$. The starting point for this article is the following observation from a previous paper by the author~\cite{twi-l2}.

\begin{lemma}
    Assume the Singer conjecture for all even-dimensional closed aspherical manifolds. Let $M$ be a closed aspherical $(2n+1)$-manifold and let $\phi\colon \pi_1(M) \surj \ZZ$ be a virtually fibered cohomology class (i.e.\ its image in a finite cover represents a fibration over the circle). Then $b_i^{(2)}(\til M; \phi) = 0$ for all $i \ne n$.
\end{lemma}
Here, the notation $b_i^{(2)}(\til M; \phi)$ stands for the \emph{twisted} $L^2$-Betti numbers of Friedl--Lück~\cite{l2thur}, which correspond, in the aspherical case, to the $L^2$-Betti numbers of the group $\ker \phi$.

As explained in~\cite[Section~7.5.1]{twi-l2}, this lemma naturally leads to a twisted version of the Singer conjecture:
\begin{conj}\label{conj:thurston5}
    Let $M$ be a closed aspherical $(2n+1)$-manifold. Then $b_i^{(2)}(\til M; {-}) \equiv 0$ for all $i \ne n$. Moreover, $\chi^{(2)}(\til M; {-}) = (-1)^n\cdot b_n^{(2)}(\til M; {-})$ is a seminorm up to sign.
\end{conj}

In order to work towards a proof, we first consider a generalization of the second part of this conjecture:

\begin{conj}\label{conj:main}
    Let $G$ be a finitely generated group satisfying the Atiyah conjecture and let $C_\blt$ be a finitely generated free $L^2$-acyclic $\ZZ G$-chain complex. Then, for every $n\ge 0$, the twisted $L^2$-Betti number $b_n^{(2)}(C_\blt; {-})$ is a polyhedral seminorm.
\end{conj}

This statement essentially appears as Conjecture~6.7 in~\cite{seminorm-conj}, and can also be considered as a generalization of Kielak's \emph{single polytope theorem}:

\begin{thm}[{see~\cite[Theorem~3.14]{kielak-bns}}]
    \label{thm:single-polytope}
    Let $K$ be a skew field and let $H$ be a finitely generated free abelian group. Let $A$ be a square matrix over the crossed product ring $K \ast H$, which embeds in its Ore division ring $\mathcal D$. Then the Newton polytope $P(A) \coloneqq P(\det \mathcal D \otimes A)$ is either empty or a single polytope.
\end{thm}
(See Section~\ref{sec:crossed-products} for a definition of \emph{crossed product ring}.) Indeed, let $G$ be a finitely generated, torsion-free group satisfying the Atiyah conjecture, with free abelianization $\ab\colon G \twoheadrightarrow H$, and let $A \in \mathrm M(n, \mathbb{Z}G)$ be a matrix that becomes invertible over the Linnell skew field $\mathcal D(G)$. Let $\mathrm{el}(r_A)$ be the $\mathbb{Z}G$-chain complex supported in degrees $0$ and $1$, with first differential $({-})\cdot A$. Then, $A$ can be regarded as a matrix over $\mathcal D(\ker \ab) \ast H$, and the thickness function of its Newton polytope is precisely the twisted $L^2$-Betti number $b_0^{(2)}(\mathrm{el}(r_A); {-})$. The single polytope implies states that this is a seminorm.

In this article, we will generalize Theorem~\ref{thm:single-polytope}, by defining a polytope $P(A)$ for every rectangular matrix $A\in \mathrm{M}(m, n, K \ast H)$ (with $m \ge n$) of full rank over $\mathcal D$. The function $b_0^{(2)}(\mathrm{el}(r_A); {-})$ can also be defined in the rectangular case, and then it equals the thickness function of $P(A)$, which is a seminorm. The full statement is given in Theorem~\ref{thm:dual-poly}.

This leads to a proof of Conjecture~\ref{conj:main}. As a consequence, we show some results about the vanishing of twisted $L^2$-Betti numbers, completing the proof of Conjecture~\ref{conj:thurston5} for a class of arithmetic hyperbolic manifolds arising from congruence subgroups:
\begin{thm}\label{thm:main-final}
    Let $M$ be a closed congruence arithmetic hyperbolic $(2n+1)$-manifold of simplest type. Then $(-1)^n \cdot \chi^{(2)}(\widetilde M; -) = b_n^{(2)}(\widetilde M; -)$ is a seminorm, and the other twisted $L^2$-Betti numbers $b_i^{(2)}(\widetilde M; -)$, with $i\ne n$, are identically zero.
\end{thm}

Our argument for the vanishing outside the lower middle dimension relies on estimating the twisted $L^2$-Betti numbers by the ordinary $L^2$-Betti numbers of totally geodesic hypersurfaces (see Proposition~\ref{prop:mayer-vietoris}), which are abundant in such manifolds by results of Bergeron--Millson--Moeglin~\cite{virt-geod}.

We note that, during the writing of this paper, a proof of Conjecture~\ref{conj:main} was given by Jaikin-Zapirain, Kudlinska and Sánchez-Peralta~\cite{jzksp}. Our proof makes use of different techniques, involving manipulations of matrices over noncommutative principal ideal domains and of their Newton polytopes, and may be of independent interest. In Theorem~\ref{thm:jzksp-polytope-equality}, we also prove that the \emph{Newton polytope} defined in~\cite{jzksp} for finitely generated $K\ast H$-modules that become trivial over $\mathcal D$ equals the polytope of Theorem~\ref{thm:dual-poly}, and is hence a single polytope. This answers~\cite[Problem~3]{jzksp} in the affirmative.

\subsection{Acknowledgments}
The author would like to thank his postdoc mentor Xiaolong Hans Han and his former doctoral advisor Bruno Martelli for helpful comments. He is also grateful to Andrei Jaikin-Zapirain for an interesting discussion.
The author is partially supported by the start-up grants at Shanghai Institute for Mathematics and Interdisciplinary Sciences.

\subsection{AI Declaration}
The idea of the proof of Theorem~\ref{thm:jzksp-polytope-equality} was obtained with GPT-6 Astra. Verification and exposition of the argument are entirely due to the author.
A few additional references were also located with Kimi K2.6 and GPT-6 Astra, and independently verified.

\subsection{Structure of the paper}
In Section~\ref{sec:prelim}, we recall preliminary notions about integral polytopes and crossed products, and introduce the \emph{greatest common right divisor} (GCRD) of a matrix, which we use in Section~\ref{sec:tri-ineq} to prove Conjecture~\ref{conj:main}. In Section~\ref{sec:dual-poly}, we give a description of a polytope dual to the twisted $L^2$-Betti number seminorm, via matrix GCRDs, which we use to prove Theorem~\ref{thm:jzksp-polytope-equality}.
Finally, in Section~\ref{sec:twisted-singer}, we prove Theorem~\ref{thm:main-final}.
\section{Preliminaries}\label{sec:prelim}
In this section, we shall recall a few technical notions that we will need in the following.

\subsection{Integral polytopes}
We start by giving a few definitions about integral polytopes, mostly following Kielak~\cite{kielak-bns}.
In the following, let $H, H'$ be finitely generated free abelian groups.
\begin{defin}
    The \emph{integral polytope monoid} $\mathfrak P(H)$ is the monoid of all (compact) polytopes in $H\otimes \mathbb R$ that are convex hulls of nonempty subsets of $H$, under the Minkowski sum. Its Grothendieck group, called the \emph{integral polytope group}, is denoted by $\mathcal P(H)$, and consists of all formal differences of integral polytopes.
\end{defin}

Since the Minkowski sum is cancellative, the integral polytope monoid embeds in its Grothendieck group: $\mathfrak P(H) \subseteq \mathcal P(H)$.

We also define $\mathfrak P_T(H)$, $\mathcal P_T(H)$ to be the quotients of $\mathfrak P(H)$, $\mathcal P(H)$ by the translation action of $H$. It is easy to see that the Minkowski sum descends to this quotient, and that $\mathcal P_T(H)$ is the Grothendieck group of $\mathfrak P_T(H)$. This gives an inclusion $\mathfrak P_T(H) \subseteq \mathcal P_T(H)$. We will call \emph{single polytopes} the elements of the images of $\mathfrak P(H), \mathfrak P_T(H)$ in $\mathcal P(H), \mathcal P_T(H)$ respectively. 


\begin{defin}
    Given $P \in \mathfrak P(H)$, define its \emph{support function} $h_P \colon \Hom(H, \RR) \to \RR$ as
    \begin{equation}
        h_P(\phi) \coloneqq \max\{\phi(x) \mid x \in P\}.
    \end{equation}
\end{defin}

The function $h_P$ is convex and positively homogeneous, and hence continuous. By~\cite[Theorem~1.7.5]{convex-bodies}, the support function is a monoid homomorphism into the set of continuous functions $\Hom(H, \RR) \to \RR$: for any two polytopes $P, Q \in \mathfrak P(H)$, we have $h_{P+Q} = h_P + h_Q$. This allows us to naturally define the support function of $X = P-Q \in \mathcal P(H)$ as $h_{X} \coloneqq h_P - h_Q$. Additionally, by~\cite[Theorem~1.7.1]{convex-bodies}, this homomorphism is injective on $\mathfrak P(H)$, which makes its extension to $\mathcal P(H)$ injective as well.

\begin{defin}
    Given $\phi\colon H \to \mathbb R$ and $P \in \mathfrak P(H)$, let $\phi[P]$ denote the length of the interval $\phi(P)$. The expression $(-)[P]$ is also called the \emph{thickness function} of $P$.
\end{defin}

The thickness function satisfies the identity $\phi[P] = h_P(\phi) + h_P(-\phi)$, making it additive in $P$. It is also well defined for polytopes in $\mathfrak P_T(H)$, and is a seminorm in $\phi$, taking integer values on integer characters. Finally, it naturally extends to $\mathcal P(H)$ and $\mathcal P_T(H)$ as a difference of two seminorms.

\begin{defin}
    Given a polytope $P \in \mathfrak P(H)$ and a character $\phi \colon H \to \mathbb R$, define the \emph{$\phi$-face} of $P$ as
    \begin{equation}
        F_\phi(P) \coloneqq \{p \in P \mid \phi(p) = \min_{q \in P} \phi(q)\}.
    \end{equation}
    This expression induces homomorphisms $\mathfrak P(H) \to \mathfrak P(H)$ and $\mathcal P(H) \to \mathcal P(H)$, which also pass to the translation quotient, and which we all denote by $F_\phi$.
\end{defin}

Any homomorphism $\omega\colon H \to H'$ induces natural maps $\mathcal P(H) \to \mathcal P(H')$ and $\mathcal P_T(H) \to \mathcal P_T(H')$, which we also denote by $\omega$. These have a well-behaved interaction with the quantities defined above:

\begin{prop}\label{prop:proj}
    Let $\omega\colon H \to H'$, let $\beta\colon H' \to \RR$, and let $P \in \mathcal P(H)$. Then:
    \begin{enumerate}[label=(\arabic*)]
        \item $h_P(\beta\circ \omega) = h_{\omega(P)}(\beta)$;
        \item $(\beta\circ \omega)[P] = \beta[\omega(P)]$;
        \item $\omega(F_{\beta\circ \omega}(P)) = F_\beta(\omega(P))$.
    \end{enumerate}
    The second and third items also hold for $P \in \mathcal P_T(H)$.
\end{prop}
\begin{proof}
    The statement follows straightforwardly from the definitions.
\end{proof}

A consequence of this result is the following uniqueness lemma.

\begin{lemma}\label{lemma:proj+face}
    Let $P, Q \in \mathcal P_T(H)$, with $\rk H\ge 2$, and let $\phi\colon H \to \ZZ$ be a nonzero character such that $F_\phi(P) = F_\phi(Q)$. Let $S \subseteq \Hom(H, \ZZ)$ be such that the union of $\vspan_\RR \{\phi, \psi\}$ for $\psi \in S$ is dense in $\Hom(H, \RR)$, and assume that for every $\psi \in S$, the characters $\phi, \psi$ factor through some full-rank map $\pi\colon H \to \ZZ^2$ such that $\pi(P) = \pi(Q)$. Then $P = Q$.
\end{lemma}
\begin{proof}
    For each $\pi$ as in the statement, let $\beta_\pi \colon \ZZ^2 \to \ZZ$ be such that $\phi = \beta_\pi \circ \pi$. 
    
    Choose representatives $P_0, Q_0 \in \mathcal P(H)$ such that $F_\phi(P_0) = F_\phi(Q_0)$ in $\mathcal P(H)$. Then $\pi(P_0) = \pi(Q_0)$ holds for every $\pi$, since by Proposition~\ref{prop:proj}~(3)
    \begin{equation}
        F_{\beta_\pi}(\pi(P_0)) = \pi(F_\phi(P_0))
        =\pi(F_\phi(Q_0)) = F_{\beta_\pi}(\pi(Q_0))
    \end{equation}
    in $\mathcal P(\ZZ^2)$, and the $\beta_\pi$-face determines a translation representative in $\mathcal P(\ZZ^2)$.

    For every $\psi \in S$, we apply Proposition~\ref{prop:proj}~(1) to the corresponding full-rank map $\pi$, obtaining $h_{P_0}(a\psi + b\phi) = h_{Q_0}(a\psi+b\phi)$ for all $a, b \in \RR$. By density and continuity of the support functions, we get $h_{P_0} \equiv h_{Q_0}$, which implies $P_0 = Q_0$.
\end{proof}
    
\begin{cor}\label{cor:proj+face}
    Let $P, Q \in \mathcal P_T(H)$, with $\rk H\ge 2$, and let $\phi\colon H \to \ZZ$ be a nonzero character. Let $S \subseteq \Hom(H, \ZZ)$ be as in Lemma~\ref{lemma:proj+face}. Assume that for every $\psi\in S$, the characters $\phi, \psi$ factor through some full-rank map $\pi\colon H \to \ZZ^2$, say $\phi = \beta_\pi \circ \pi$, such that $\pi(P) - \pi(Q) \in \mathcal P_T(\ker \beta_\pi)$. Then $P - Q \in \mathcal P_T(\ker \phi)$.
\end{cor}
\begin{proof}
    Let $Q' \coloneqq Q - F_\phi(Q)+ F_\phi(P)$. It suffices to prove that $P = Q'$. For any $\psi \in S$ with corresponding $\pi$, we have
    \begin{equation}
        \pi(P) - \pi(Q) = F_{\beta_\pi}(\pi(P) - \pi(Q)) = \pi(F_\phi(P) - F_\phi(Q)) = \pi(Q') - \pi(Q),
    \end{equation}
    so $\pi(P) = \pi(Q')$. The conclusion follows from Lemma~\ref{lemma:proj+face}.
\end{proof}

\begin{defin}[{\cite[Definition~4.2]{funke}}]
    Let $H$ be a finitely generated free abelian group. A subgroup $K \le H$ is \emph{pure} if it is the intersection of $H$ with a vector subspace of $H \otimes \RR$, or equivalently, if it is a direct summand of $H$.
\end{defin}
We notice that the intersection of a family of pure subgroups is again a pure subgroup.
\begin{prop}\label{prop:inters}
    Let $K, K'$ be two pure subgroups of $H$. Then $\mathcal P_T(K)$, $\mathcal P_T(K')$ and $\mathcal P_T(K\cap K')$ naturally embed into $\mathcal P_T(H)$, and $\mathcal P_T(K) \cap \mathcal P_T(K') = \mathcal P_T(K\cap K')$.
\end{prop}
\begin{proof}
    By~\cite[Theorem~4.5]{funke}, the group $\mathcal P_T(H)$ is free abelian with a basis $\mathcal B$ of single polytopes in $\mathfrak P_T(H)$, and for every pure subgroup $L \le H$, $\mathcal B \cap \mathfrak P_T(L) = \mathcal B \cap \mathcal P_T(L)$ is a basis of $\mathcal P_T(L)$. The inclusion $\mathcal P_T(K\cap K') \subseteq \mathcal P_T(K) \cap \mathcal P_T(K')$ is clear. Conversely,
    \begin{align}
        \mathcal P_T(K) \cap \mathcal P_T(K') &= \operatorname{span}_\ZZ(\mathcal B \cap \mathfrak P_T(K) \cap \mathfrak P_T(K') )
        \\ &= \operatorname{span}_\ZZ(\mathcal B \cap \mathfrak P_T(K \cap K')) = \mathcal P_T(K\cap K').
    \end{align}
\end{proof}

\begin{defin}
    A polytope $P\in \mathfrak P_T(H)$ is called \emph{$\phi$-flat} if $\phi[P] = 0$.
\end{defin}
Note that a polytope in $\mathfrak P_T(H)$ is $\phi$-flat if and only if it is an element of $\mathfrak P_T(\ker \phi)$, while an element of $\mathcal P_T(H)$ is a difference of $\phi$-flat polytopes if and only if it is an element of $\mathcal P_T(\ker \phi)$.

\begin{defin}
    We define a partial order on $\mathfrak P_T(H)$, where $P \le Q$ whenever there exists a polytope $R \in \mathfrak P_T(H)$ such that $P + R = Q$.
\end{defin}

It is easy to see that, for any $\phi \colon H\otimes \mathbb R \to \mathbb R$ and any $P, Q \in \mathfrak P_T(H)$, $P \le Q$ implies $F_\phi(P) \le F_\phi(Q)$ and $\phi[P] \le \phi[Q]$. Let us now prove a technical lemma. 

\begin{lemma}\label{lemma:magic}
    Let $P \in \mathfrak P_T(\ZZ^2)$ be a polytope and let $\phi \colon \ZZ^2 \to \ZZ$ be a character. Let $S$ be the set of all primitive characters in $\Hom(\ZZ^2, \ZZ)$. Then
    \begin{equation}
        \phi[P] = \frac 12 \sum_{\alpha \in S} \phi[F_\alpha(P)].
    \end{equation}
\end{lemma}
\begin{proof}
    The formula is easily shown to hold if $P$ is a point or segment. Otherwise, the boundary of $P$ consists exactly of the faces $F_\alpha(P)$ for $\alpha \in S$. If $\phi = 0$, then the statement is trivial; otherwise it follows from the simple observation that the map $\phi \colon \partial P \to \phi(P)$ is a double cover on the interior of the segment $\phi(P)$.
\end{proof}

This lemma is at the core of the following result, which will in turn enable the proof of the main theorem.

\begin{prop}\label{prop:main}
    Let $\poly D, \poly E, \poly B_1, \poly B_2 \in \mathfrak P_T(\ZZ^2)$. Let $a_1, a_2 \in \ZZ_{>0}$, and let $\phi_1, \phi_2\colon \ZZ^2 \to \mathbb Z$ be linearly independent characters. Let $a_1\phi_1 + a_2\phi_2 = c\phi$, with $c \in \ZZ_{>0}$ and $\phi$ primitive. Assume that:
    \begin{itemize}
        \item $\poly E$ is $\phi$-flat;
        \item $\poly D \le \poly B_i + \poly E$ for $i = 1, 2$.
    \end{itemize}
    Then $c\phi[\poly D] \le a_1\phi_1[\poly B_1] + a_2\phi_2[\poly B_2]$.
    
    If equality holds, then $\poly B_i - \poly D + F_\phi(\poly D) \in \mathfrak P_T(\ker \phi_i)$ for $i = 1, 2$.
\end{prop}
\begin{proof}
As in Lemma~\ref{lemma:magic}, let $S$ be the set of all primitive characters in $\Hom(\ZZ^2, \ZZ)$, and let $Z \coloneqq S \setminus \{\pm \phi\}$. Define also $P_i \coloneqq \poly B_i + \poly E - \poly D \in \mathfrak P_T(\ZZ^2)$. Recall that $\phi[F_{\pm \phi}(Q)] = 0$ for any $Q \in \mathcal P_T(\ZZ^2)$, and notice that $F_\alpha(\poly E)$ is a singleton for $\alpha \in Z$. Then, by Lemma~\ref{lemma:magic}:
{\allowdisplaybreaks
\begin{align}
2c\phi[\poly D] &= \sum_{\alpha \in S}c\phi[F_\alpha(\poly D)]
\\ &= \sum_{\alpha \in Z}c\phi[F_\alpha(\poly D)]
\\ &\le \sum_{\alpha \in Z} \sum_{i = 1}^2
a_i\phi_i[F_\alpha(\poly D)]
\\ &\le \sum_{\alpha \in Z} \sum_{i = 1}^2
\bigg(a_i\phi_i[F_\alpha(\poly D)] + a_i\phi_i[F_\alpha(P_i)]\bigg)
\\ &= \sum_{\alpha \in Z} \sum_{i = 1}^2
a_i\phi_i[F_\alpha(\poly D + P_i)]
\\ &= \sum_{\alpha \in Z} \sum_{i = 1}^2
a_i\phi_i[F_\alpha(\poly B_i + \poly E)]
\\ &= \sum_{\alpha \in Z} \sum_{i = 1}^2
\bigg(a_i\phi_i[F_\alpha(\poly B_i)] + a_i\phi_i[F_\alpha(\poly E)]\bigg)
\\ &= \sum_{\alpha \in Z} \sum_{i = 1}^2
a_i\phi_i[F_\alpha(\poly B_i)]
\\ &\le \sum_{\alpha \in S} \sum_{i = 1}^2
a_i\phi_i[F_\alpha(\poly B_i)]
\\ &= 2 \sum_{i=1}^2 a_i\phi_i[\poly B_i].
\end{align}}

The inequality is proved; let us now assume that equality holds. Then, in particular, we have $\phi_1[F_\phi(\poly B_1)] = \phi_2[F_\phi(\poly B_2)]  = 0$, and $\phi_1[F_\alpha(P_1)] = \phi_2[F_\alpha(P_2)] = 0$ for all $\alpha \in Z$.

Notice that every edge of $P_i$ is $\phi$-flat or $\phi_i$-flat. Indeed, if an edge $e$ of $P_i$ is not $\phi$-flat, then it is $F_\alpha(P_i)$ for some $\alpha\in Z$, and hence it is $\phi_i$-flat.
It follows that $P_i$ is a (possibly degenerate) parallelogram and can be written as $Q_i + R_i$, where $Q_i = F_\phi(P_i) \in \mathfrak P_T(\ker \phi)$ and $R_i = F_{\phi_i}(P_i)\in \mathfrak P_T(\ker \phi_i)$. Hence
\begin{equation}
    \poly D + Q_i + R_i = \poly B_i + \poly E.
\end{equation}
By taking $F_\phi$, we also have
\begin{equation}
    F_\phi(\poly D) + Q_i = F_\phi(\poly B_i) + \poly E.
\end{equation}
Subtracting, we obtain
\begin{equation}
    \poly B_i - \poly D + F_\phi(\poly D) = R_i + F_\phi(\poly B_i).
\end{equation}
Since $R_i \in \mathfrak P_T(\ker \phi_i)$ and $F_\phi(\poly B_i)\in\mathfrak P_T(\ZZ^2)$ is $\phi_i$-flat, the proof is complete.
\end{proof}

\subsection{Crossed products and skew polynomial rings}\label{sec:crossed-products}
We will now introduce a construction generalizing the usual group ring.

\begin{defin}[{\cite[Section~10.3.2]{l2}}]
    Let $R$ be a ring and let $G$ be a group. Consider set maps $c\colon G \to \operatorname{Aut}(R)$, $c\colon g \mapsto c_g$, and $\tau\colon G\times G \to R^\times$, such that:
    \begin{align}
        c_g (c_{g'}(r)) &= \tau(g,g') \cdot c_{gg'}(r)\cdot \tau(g,g')^{-1},
        \\ \tau(g, g') \cdot \tau(gg', g'') &= c_g(\tau(g', g'')) \cdot \tau (g, g'g''),
    \end{align}
    for all $g,g',g'' \in G$ and $r \in R$.    
    The \emph{crossed product} $R \ast G = R \ast_{c,\tau} G$ is a ring, consisting of all formal $R$-linear combinations of elements of $G$, with multiplication given by the rules $g \cdot r = c_g (r) \cdot g$ and $g \cdot g'$ = $\tau(g, g') \cdot (gg')$. 
\end{defin}
The usual group ring is recovered in the special case $c_g \equiv \id_R$, $\tau \equiv 1$. If $\tau$ is trivial and $G\simeq \ZZ$, it is easy to see that we obtain a so-called \emph{skew Laurent polynomial ring}:
\begin{defin}
    Let $R$ be a ring with an automorphism $\sigma$. The \emph{skew Laurent polynomial ring} $R[u^\pm; \sigma]$ (or simply $R[u^\pm]$) is the ring of all finite sums of monomials $ru^\alpha$, with $r \in R$ and $\alpha \in \mathbb Z$, and with multiplication given by the rule $u r = \sigma(r) u$.
\end{defin}
Indeed, by~\cite[Example~10.54]{l2}, every crossed product with $\ZZ$ is isomorphic to a skew Laurent polynomial ring, with the variable $u$ corresponding to a generator of $\ZZ$. We also mention the following important property:
\begin{prop}[{\cite[Proposition~2.1.1~(iii)]{cohn}}]
If $K$ is a skew field, the skew Laurent polynomial ring $K[u^\pm]$ is a noncommutative principal ideal domain: every left or right ideal is principal.
\end{prop}

If $K$ is a skew field and $H$ is a finitely generated free abelian group, then $K\ast H$ is an \emph{Ore domain} and hence admits a skew field of fractions $\mathcal D\coloneqq \operatorname{Ore}(K\ast H)$~\cite[Theorem~2.14]{kielak-bns}. By Lemma~3.12 in~\cite{kielak-bns} and the discussion that follows, there is a well-defined \emph{polytope homomorphism} $P\colon K\ast H\setminus \{0\} \to \mathfrak P(H)$, a monoid homomorphism sending each element to the convex hull of its support. This map extends to a group homomorphism $P\colon \mathcal D^\times \to \mathcal P(H)$.


\subsection{Matrix greatest common divisors}

Let $R$ be any ring and let $M \in \mathrm{M}(m, n, R)$ be a rectangular matrix with $m\ge n$. We adopt the convention where matrices act by multiplication on the right, and modules are left modules, unless otherwise specified.

Recall that, if $R$ is a principal ideal domain (PID), a greatest common right divisor (GCRD) of a tuple $A = (A_1, \dots, A_m) \in R^m$ is a generator of the left ideal $^\bullet(A_1, \dots, A_m)$. Hence, if $A$ is considered as a column vector, we can say:

\begin{defin}
    An element $d \in R$ is a \emph{greatest common right divisor} of $A \in \mathrm{M}(m,1,R)$ if there exist vectors $S \in \mathrm{M}(1,m,R), T \in \mathrm{M}(m,1,R)$ such that $SA = d$ and $Td = A$.
\end{defin}

We can generalize this definition as follows.

\begin{defin}
    A square matrix $D \in \mathrm{M}(n,R)$ is a \emph{greatest common right divisor} of $A \in \mathrm{M}(m,n,R)$ if there exist matrices $S \in \mathrm{M}(n,m,R), T \in \mathrm{M}(m,n,R)$ such that $SA = D$ and $TD = A$.
\end{defin}

We note that this concept has appeared in the symbolic computation literature in various forms, with $R$ a commutative~\cite{gcrd-bl, gcrd-nv} or skew~\cite{gcrd-bcl} polynomial ring.

Let us now discuss existence and uniqueness of the matrix GCRD.

\begin{lemma}
    If $R$ is a PID and $m \ge n$, then every matrix $A \in \mathrm{M}(m,n,R)$ has a GCRD.
\end{lemma}
\begin{proof}
    Let $A = UJV$ be the Jacobson normal form of $A$, with:
    \begin{itemize}
        \item $U \in \mathrm{GL}(m, R)$;
        \item $V \in \mathrm{GL}(n, R)$;
        \item $J$ diagonal, with its last $m-n$ rows zero.
    \end{itemize}
    Then, it suffices to take:
    \begin{itemize}
        \item $D$ as the first $n$ rows of $JV$;
        \item $T$ as the first $n$ columns of $U$;
        \item $S$ as the first $n$ rows of $U^{-1}$.
    \end{itemize}\vspace{-4.5ex}
\end{proof}
\begin{lemma}\label{lemma:hermite}
    Let $R$ be a PID and let $m\ge n$. Given two matrices $A \in M(m,n,R)$ and $D \in M(n,R)$, we have $D = \gcrd_R A$ if and only if $\image A = \image D$.
\end{lemma}
\begin{proof}
    If there exist $S, T$ such that $SA = D$ and $TD = A$, then clearly $\image A = \image D$. Conversely, if $A$ and $D$ have the same image, then they have the same Hermite normal form up to adding zero rows. That is, there exist matrices $U \in \mathrm{GL}(m,R), U' \in \mathrm{GL}(n,R)$ such that
    \begin{equation}
        UA = \begin{bmatrix}
            U'D \\ 0
        \end{bmatrix}.
    \end{equation}
    We can now take $S$ as $U'^{-1}$ times the first $n$ rows of $U$, and $T$ as the first $n$ columns of $U^{-1}$ times $U'$.
\end{proof}
For general rings, we have the following results.

\begin{lemma}\label{lemma:gcrd-integer}
    Let $D$ be a GCRD of a matrix $A \in \mathrm{M}(m,n,R)$. If there exist two matrices $Z, Z' \in \mathrm{M}(n, R)$ such that $D = Z'ZD$, then $D' \coloneqq ZD$ is also a GCRD of $A$.
\end{lemma}
\begin{proof}
    If $SA = D$ and $TD = A$, then $ZSA = D'$ and $TZ'D' = A$.
\end{proof}
A consequence of this lemma is that left multiplication by an invertible matrix preserves the GCRD property. Let us now show a converse to it.

\begin{lemma} \label{lemma:unique-gcrd}
    Let $D, D'$ be two GCRDs of a matrix $A \in \mathrm{M}(m,n,R)$. Then there exist two matrices $Z, Z' \in \mathrm{M}(n, R)$ such that $D = Z'D'$ and $D' = ZD$.
\end{lemma}
\begin{proof}
    There exist matrices $S, S', T, T'$ such that $SA = D, TD = A, S'A = D', T'D' = A$. Hence, we have $ST'D' = D$ and $S'TD = D'$, and we can define $Z\coloneqq S'T, Z'\coloneqq ST'$.
\end{proof}

\begin{remark}
    From now on, we will adopt the notation $\gcrd_R A$ to denote any GCRD of the matrix $A$ over the ring $R$. When this notation appears in a statement, it shall be taken to mean that the statement is true for every such choice of GCRD, unless otherwise specified.
\end{remark}

\section{Dimension of the cokernel}\label{sec:tri-ineq}

Assuming the Atiyah conjecture, we may rephrase Conjecture~\ref{conj:main}. Let $\mathcal D$ denote the Linnell skew field $\mathcal D(G)$, and given any primitive character $\phi$ defined on $G$, let $D_\phi$ denote $\mathcal D(\ker \phi)$. We would like to show that the map
\begin{equation}
    \phi \mapsto \dim_{D_\phi}(H_n(D_\phi[u^\pm] \otimes_{\mathbb Z G} C_\blt)) \in \mathbb N,
\end{equation}
extended homogeneously to all integer characters in $\Hom(G, \mathbb Z)$, is an integer seminorm. In fact, we only need to prove the triangle inequality~\cite{thurston,kapovich}.

Let $\ab \colon G \twoheadrightarrow H$ be the universal free abelian quotient. The cases $\rk H = 0,1$ are trivial, so let us assume $\rk H \ge 2$, and let $K \coloneqq \mathcal D(\ker \ab)$. By~\cite[Section~6.1]{l2thur}, the skew field $\mathcal D$ is the Ore skew field of fractions of a crossed product ring $R \coloneqq K \ast H$. Analogously, $D_\phi \simeq \operatorname{Ore}(K \ast \ker \phi)$.

Since every skew Laurent polynomial ring $R_\phi \coloneqq D_\phi[u^\pm]$ (where $\phi(u) = 1$) is a localization of $R$, and localization is exact~\cite[Proposition~II.3.5]{ore-flat}, we may define $M \coloneqq H_n( R \otimes_{\mathbb Z G} C_\blt)$ and prove that
\begin{equation}
    \dim_{D_\phi}(R_\phi \otimes_{R} M)
\end{equation}
is a seminorm in $\phi$. For the sake of convenience, we also define $R_{k\phi} \coloneqq R_\phi$ for $\phi$ primitive and $k \in \ZZ \setminus \{0\}$.

By~\cite[Lemma~10.55~(2)]{l2}, $R$ is Noetherian, so $M$ is a finitely presented module, isomorphic to the cokernel of a matrix $A\in \mathrm{M}(m, n, R)$ (acting by right multiplication, where $m \ge n$), with full rank over $\mathcal D$. Hence, we need to prove that the function
\begin{equation}
    v_A(\phi)\coloneqq \dim_{D_\phi} \coker(R_\phi\otimes_R A),
\end{equation}
extended homogeneously to all integer characters $\phi$, satisfies the triangle inequality.

Consider primitive characters $\phi_1 \ne \pm \phi_2$ and positive integers $a_1,a_2$. Let $a_1\phi_1 + a_2\phi_2 = c\phi$, with $\phi$ primitive and $c \ge 1$.
We shall prove the triangle inequality $v_A(a_1\phi_1 + a_2\phi_2) \le a_1v_A(\phi_1) + a_2v_A(\phi_2)$.

Let $B_1$ be a GCRD of $A$ over $R_{\phi_1}$, with $S_1A = B_1$; up to left multiplication, we can clear denominators in both $S_1$ and $B_1$, using Lemma~\ref{lemma:gcrd-integer}, and assume that $S_1$ and $B_1$ have entries in $R$. Then, since $\coker (R_{\phi_1}\otimes B_1) = \coker (R_{\phi_1} \otimes A)$, we have $a_1v_A(\phi_1) = a_1\deg_{\phi_1}(\det B_1) = a_1\phi_1[P(B_1)]$.

Analogously, we can find a GCRD of $A$ over $R_{\phi_2}$, say $B_2$, such that $S_2A = B_2$, with $S_2$ and $B_2$ having entries in $R$; again, we have $a_2v_A(\phi_2) = a_2\phi_2[P(B_2)]$.

Finally, construct a block matrix $B\in \mathrm{M}(2n, n, R)$ by stacking $B_1$ on top of $B_2$, and define $D \in \mathrm{M}(n, R)$ as a GCRD of $B$ over $R_\phi$, with $D = SB$ (and $S \in \mathrm{M}(n,2n,R)$). Now
\begin{equation}
    D = SB = S\begin{bmatrix}
        B_1 \\ B_2
    \end{bmatrix} = S\begin{bmatrix}
        S_1 \\ S_2
    \end{bmatrix} A,
\end{equation}
and so $\image_{R_\phi} (D) \subseteq \image_{R_\phi} (A)$. It follows that:
\begin{align}
    v_A(a_1\phi_1 + a_2\phi_2) = cv_A(\phi) 
   & = c\dim_{D_{\phi}} \coker (R_{\phi} \otimes A) 
\\ & \le c\dim_{D_{\phi}} \coker (R_{\phi} \otimes D) \label{eq:ineq-hermite}
\\ &= c \phi[P(\det D)]
\\ &= (a_1\phi_1 + a_2\phi_2)[P(\det D)].
\end{align}

Now recall that there exists a matrix $T \in \mathrm{M}(2n, n, R_{\phi})$ such that $TD = B$. We can clear the denominators in $T$ by multiplying both sides by a suitable element $e \in K \ast \ker \phi$ on the left. Hence, we get
\begin{equation}
    \begin{bmatrix}
        eT_1 \\ eT_2
    \end{bmatrix} D = \begin{bmatrix}
        eB_1 \\ eB_2
    \end{bmatrix}.
\end{equation}
Note that $T_1, T_2, D, B_1, B_2$ are invertible over $\mathcal D$. Taking the Newton polytope for each block, we find:
\begin{equation}
    \begin{cases}
        P(D) \le P(B_1) +P(E),\\
        P(D) \le P(B_2) +P(E),
    \end{cases}
\end{equation}
where $E$ is the matrix $e I_n$. Recall that, by the single polytope theorem, all polytopes involved here are single polytopes in $\mathfrak P_T(H)$. Let $\pi = (\phi_1, \phi_2) \colon H \to \ZZ^2$, and let $j_1, j_2 \colon \ZZ^2 \twoheadrightarrow \ZZ$ be the two coordinate projections, so that $\phi_i = j_i \circ \pi$. Then, since $P(E)$ is $(a_1\phi_1+a_2\phi_2)$-flat, $\pi(P(E))$ is $(a_1 j_1 + a_2 j_2)$-flat, and we can apply Proposition~\ref{prop:main} with $\poly D = \pi(P(D))$, $\poly E = \pi(P(E))$, $\poly B_1 = \pi(P(B_1))$, $\poly B_2= \pi(P(B_2))$, and obtain
\begin{align}
    v_A(a_1\phi_1+a_2\phi_2) 
    &\le (a_1\phi_1 + a_2\phi_2)[P(D)]
\\  &= (a_1 j_1 + a_2 j_2)[\poly D]
\\  &\le a_1 j_1[\poly B_1] + a_2 j_2[\poly B_2]
\\  &= a_1\phi_1[P(B_1)] + a_2\phi_2[P(B_2)] 
\\  &= a_1v_A(\phi_1) + a_2v_A(\phi_2)
\end{align}
for $\phi_1 \ne \pm \phi_2$ primitive and $a_1,a_2$ positive integers. Clearly, the inequality holds for $\phi_1 = \pm \phi_2$ as well, and can be extended to real characters in $\Hom(G, \mathbb R)$, yielding a polyhedral seminorm (see~\cite{thurston,kapovich}). 

In summary, we have proved:
\begin{thm}\label{thm:seminorm}
    Let $K$ be a skew field, let $H$ be a finitely generated free abelian group, and let $R$ be a crossed product ring $K\ast H$. Given a matrix $A\in \mathrm{M}(m, n, R)$, where $m \ge n$, having full rank over the Ore skew field of fractions, the function
    \begin{equation}
        v_A(\phi)\coloneqq \dim_{D_\phi} \coker(R_\phi\otimes_R A),
    \end{equation}
    defined over primitive integer characters $\phi\colon H \twoheadrightarrow \mathbb Z$,
    extends by homogeneity and continuity to a polyhedral seminorm over $\Hom(H, \mathbb R)$.
\end{thm}
This result directly implies the initial conjecture, which we restate here.
\begin{cor}[{Conjecture~\ref{conj:main}}]\label{cor:main}
    Let $G$ be a finitely generated group satisfying the Atiyah conjecture and let $C_\blt$ be a finitely generated free $L^2$-acyclic $\ZZ G$-chain complex. Then, for every $n\ge 0$, the twisted $L^2$-Betti number $b_n^{(2)}(C_\blt; {-})$ is a polyhedral seminorm.
\end{cor}

This result was also proved in a recent paper of Jaikin-Zapirain, Kudlinska and Sánchez-Peralta~\cite{jzksp}. Their approach is to pass to a rank-two crossed product, and then take the quotient of the localized module by the maximal submodule of finite dimension over the coefficient field. The resulting quotient has projective dimension at most $1$ and hence admits a square presentation matrix $A$, to which Kielak's single polytope theorem can be applied.

Our proof instead establishes the triangle inequality via GCRDs of a given presentation matrix in rank-two projections. In the next section, we shall use this framework to construct two-dimensional single polytopes in such projections, which can be combined into a single polytope in $\mathfrak P_T(H)$: see Theorem~\ref{thm:dual-poly}. By the $K$-theoretic argument of Theorem~\ref{thm:jzksp-polytope-equality}, this coincides with the module Newton polytope of~\cite[Section~4.2]{jzksp}, providing an alternate proof of the former theorem.



\section{The dual polytope}\label{sec:dual-poly}

In order to understand the combinatorics of the seminorm $v_A$, we study the equality case of the triangle inequality. This will lead to a proof of the following result.

\begin{thm}\label{thm:dual-poly}
    Let $K$ be a skew field, let $H$ be a finitely generated free abelian group, and let $R$ be a crossed product ring $K\ast H$ with skew field of fractions $\mathcal D$. Let $A\in \mathrm{M}(m, n, R)$ (with $m \ge n$) be a matrix of full rank over $\mathcal D$. Then there exists a polytope $P(A) \in \mathfrak P_T(H)$, whose thickness function is the seminorm $v_A$ of Theorem~\ref{thm:seminorm}, and such that for every primitive character $\phi\colon H \twoheadrightarrow \mathbb Z$, we have
    \begin{equation}
        P(\gcrd_{R_\phi} A) = P(A) + E
    \end{equation}
    for some $E \in \mathcal P_T(\ker \phi)$. 
\end{thm}

Given a primitive character $\alpha\colon H \to \mathbb Z$, define $f_A(\alpha) \subseteq \mathcal P_T(H)$ as
\begin{equation}
    f_A(\alpha) \coloneqq P(\gcrd_{R_\alpha} A) + \mathcal P_T(\ker \alpha).
\end{equation}
This expression is well defined: by Lemma~\ref{lemma:unique-gcrd}, for any two GCRDs $D, D'$ of $A$ over $R_\alpha$, there exist $Z, Z' \in \mathrm{M}(n,R_\alpha)$ such that $D' = ZD$, $D = Z'D'$. As $D$ has the same image as $A$ over $R_\alpha$, it is invertible over $\mathcal D$ and hence $Z'Z = I_n$. By clearing denominators and applying the single polytope theorem, we find $P(Z), P(Z') \in \mathcal P_T(\ker \alpha)$.

As a consequence, with a slight abuse of notation, we have the formula
\begin{equation}\label{eq:formula}
    v_A(\alpha) = \alpha[f_A(\alpha)]
\end{equation}
for all primitive characters $\alpha$.

\subsection{The two-dimensional case}

Notice that Theorem~\ref{thm:dual-poly} is trivial for $\rk H \le 1$, and that the statement of Proposition~\ref{prop:main} is two-dimensional. Hence, we assume $\rk H \ge 2$ and consider an arbitrary full-rank map $\pi \colon H \to \ZZ^2$. Let us define $v_A^{\pi}\colon \Hom(\ZZ^2, \ZZ) \to \ZZ_{\ge 0}$ by $v_A^{\pi}(\beta) \coloneqq v_A(\beta \circ \pi)$, and given $\beta \colon \ZZ^2 \to \ZZ$ nonzero, define with abuse of notation as before
\begin{equation}\label{eq:fA-pi-def}
     f_A^\pi(\beta) \coloneqq \pi(P(\gcrd_{R_{\beta \circ \pi}} A)) + \mathcal P_T(\ker \beta) \subseteq \mathcal P_T(\ZZ^2).
\end{equation}
Then $v_A^{\pi}$ is an integer seminorm and
\begin{equation}
    v_A^{\pi}(\beta) = \beta[f_A^{\pi}(\beta)],
\end{equation}
which is again a well-defined expression. We shall now examine the unit ball of $v_A^{\pi}$.

We will call \emph{cones} of a seminorm $v$ the maximal regions of its domain where it is a linear function. In the case where $v$ is a norm, these are the closed cones over the facets of the unit ball. In the degenerate case $v \equiv 0$, the only cone is the whole domain, and if $\operatorname{codim} \ker v = 1$, then the cones are two half-spaces separated by the hyperplane $\ker v$. Finally, by \emph{open cone} we shall mean the interior of a cone.

Consider now two linearly independent primitive characters $\beta_1, \beta_2 \colon \ZZ^2 \to \mathbb Z$ in the same cone of $v_A^\pi$, let $a_1, a_2$ be positive integers, and let $a_1\beta_1 + a_2\beta_2 = c\beta$ for $c \ge 1$ and $\beta$ primitive. Then we have $cv_A^\pi(\beta) = a_1v_A^\pi(\beta_1) + a_2v_A^\pi(\beta_2)$. Let $\phi = \beta \circ \pi$ and $\phi_i = \beta_i\circ \pi$ be the corresponding characters on $H$. By the construction of the previous section, we can find $B_i \in \mathrm{M}(n, R)$ GCRDs of $A$ over the rings $R_{\phi_i}$, then construct a block matrix $B \in \mathrm{M}(2n, n, R)$, and define $D \in \mathrm{M}(n, R)$ as a GCRD of $B$ over $R_\phi$.

Then, in inequality~(\ref{eq:ineq-hermite}), equality holds: $\dim_{D_\phi} \coker({R_\phi} \otimes A) = \dim_{D_\phi} \coker({R_\phi} \otimes D)$. Since it is known that $\image_{R_\phi}(D) \subseteq \image_{R_\phi}(A)$, we deduce $\image_{R_\phi}(D) = \image_{R_\phi}(A)$. Hence, by Lemma~\ref{lemma:hermite}, $D$ is a GCRD of $A$ over $R_\phi$.

We shall denote by $\poly D, \poly B_i \in \mathfrak P_T(\ZZ^2)$ the projections of the Newton polytopes $P(D), P(B_i)$ through $\pi$. By the equality conditions of Proposition~\ref{prop:main} and by Proposition~\ref{prop:inters}, we have
\begin{equation}
\poly D - F_\beta(\poly D) \in \poly B_i + \mathcal P_T(\ker \beta_i) = f_A^{\pi}(\beta_i)    ,
\end{equation}
\begin{equation}
    \mathcal P_T(\ker \beta_1) \cap \mathcal P_T(\ker \beta_2) = \mathcal P_T(\ker \beta_1 \cap \ker \beta_2) = \mathcal P_T(0) = 0.
\end{equation}
Hence, there exists a unique element
\begin{equation}
    \poly D - F_\beta(\poly D) \in f_A^\pi(\beta) \cap f_A^\pi(\beta_1) \cap f_A^\pi(\beta_2).
\end{equation}
If $C$ is a cone of $v_A^\pi$, it is not hard to see that, by varying the coefficients $a_i > 0$ and the characters $\beta_i$ within $C$, the intersection of all $f_A^\pi(\alpha)$ for $\alpha$ primitive in $C$ contains a unique polytope $P_C \in \mathcal P_T(\ZZ^2)$. 

\begin{lemma} \label{lemma:2d-cones}
    $P_C$ does not depend on $C$.
\end{lemma}
\begin{proof}
Note that if $C$ is a cone of $v_A^\pi$, then $-C$ is also a cone, and $P_C = P_{-C}$ by~\eqref{eq:fA-pi-def}.

It suffices to prove that if $C_1 \ne C_2$ are \emph{adjacent} (that is, if they intersect outside the origin), then $P_{C_1} = P_{C_2}$. To this end, consider a sequence of pairwise distinct cones $C_1, C_2, \dots, C_{s+1}$ such that $C_i, C_{i+1}$ are adjacent for $i = 1, \dots, s$ and $C_{s+1} = -C_1$.

Let $\alpha_i$ be a primitive character in $C_i\cap C_{i+1}$ (which is a rational ray or line in $\RR^2$). Now $\Delta_i \coloneqq P_{C_{i+1}}-P_{C_{i}}$ is an element of $\mathcal P_T(\ker \alpha_i)$. The characters $\alpha_i$ are pairwise non-proportional for $i = 1, \dots, s$. This implies that $F_{\alpha_1}\colon \mathcal P_T(\ker \alpha_i) \to \mathcal P_T(\ker \alpha_i)$ is the identity for $i=1$ and zero for $i > 1$. As noted before, we also have $P_{C_{s+1}} = P_{-C_1} = P_{C_1}$. Hence, $\Delta_1 + \dots + \Delta_s = P_{C_{s+1}} - P_{C_1} = 0$, 
and by applying $F_{\alpha_1}$, we obtain $P_{C_2} - P_{C_1} = \Delta_1 = 0$.
\end{proof}

Thus, for all full-rank $\pi \colon H \to \ZZ^2$, there exists a polytope $P_\pi(A) \in \mathcal P_T(\ZZ^2)$ such that
\begin{equation} \label{eq:P-pi-unique-inters}
    \bigcap \left\{f_A^\pi(\beta) \mid \beta\colon \ZZ^2 \to \ZZ \text{ primitive}\right\} = \{P_\pi(A)\},
\end{equation}
and $v_A^\pi(\beta) = \beta[P_\pi(A)]$ for all $\beta \colon \ZZ^2 \to \ZZ$. Since $P_\pi(A)\in \mathfrak P_T(\ZZ^2) + \mathcal P_T(\ker \beta)$, by~\cite[Lemma~4.7]{funke-am}, $P_\pi(A)$ is actually a single polytope in $\mathfrak P_T(\ZZ^2)$. It is not hard to see from the characterization~\eqref{eq:P-pi-unique-inters} that, if $\theta\colon \ZZ^2\to \ZZ^2$ is full-rank, then $P_{\theta \circ \pi}(A) = \theta(P_\pi(A))$.

\subsection{Combining rank-two projections}

For each primitive character $\phi\colon H \twoheadrightarrow \ZZ$, choose a GCRD $D^\phi \in \mathrm{M}(n, R)$ of $A$ over $R_\phi$, and let $Q^\phi \coloneqq P(D^\phi)$, $P^\phi \coloneqq Q^\phi - F_\phi(Q^\phi) \in \mathcal P_T(H)$. Then $F_\phi(P^\phi)$ is a singleton; let $P^\phi_0 \in \mathcal P(H)$ be the unique representative of $P^\phi$ such that $F_\phi(P^\phi_0) = \{0\} \subseteq H\otimes \RR$, and let $h^\phi$ be the support function of $P^\phi_0$.

\begin{lemma}
    As $\phi$ varies within an open cone $C$ of $v_A$, the support function $h^\phi$ does not depend on $\phi$.
\end{lemma}
\begin{proof}
    If $\phi\in C$, and $\pi\colon H \to \ZZ^2$ is a full-rank map such that $\phi = \beta \circ \pi$ for some $\beta$, then $\beta$ is in an open cone of $v_A^\pi$, and hence $F_\beta(P_\pi(A))$ is a singleton. As such, we have $\pi(P^\phi) = P_\pi(A)$ for all $\pi$ through which $\phi$ factors.

    Let $\psi \ne \pm \phi$ be a primitive character in some open cone of $v_A$, and let $\pi = (\psi, \phi)$, with $\alpha, \beta\colon \ZZ^2 \twoheadrightarrow \ZZ$ the two coordinate projections. Then $F_\beta(P_\pi(A))$ is a singleton; on the other hand, 
    \begin{equation}
        F_\beta(P_\pi(A)) = F_\beta(\pi(Q^\psi - F_\psi(Q^\psi))) = \pi(F_\phi(Q^\psi)) - \pi(F_\phi(F_\psi(Q^\psi))).
    \end{equation}
    It follows that $F_\phi(Q^\psi)$ is $\psi$-flat; fixing a translation representative $Q^\psi_0$ of $Q^\psi$ such that $\psi(F_\psi(Q^\psi_0)) = 0$, the expression $z(\psi,\phi) \coloneqq \psi(F_\phi(Q^\psi_0))$ is a well-defined integer.

    Let $\phi_1, \phi_2 \in C$ be primitive. As $\phi$ varies in the closed convex cone generated by $\phi_1, \phi_2$, the faces $F_\phi(Q^\psi_0)$ form a connected polyhedral complex on which $\psi$ is constant. Hence, $z(\psi, -)$ is constant on this convex cone, and we have $h^{\phi_1}(-\psi) = z(\psi, \phi_1) = z(\psi, \phi_2) = h^{\phi_2}(-\psi)$. By the usual density argument on $\psi$, this finishes the proof.
\end{proof}

Hence, for an open cone $C$, there exists a support function $h^C$, corresponding to a polytope $P^C_0 \in \mathcal P(H)$, such that for all primitive characters $\phi \in C$ we have $h^\phi = h^C$ and $P^\phi_0 = P^C_0$. In the translation quotient, the corresponding polytope $P^C \in \mathcal P_T(H)$ satisfies
\begin{equation}
    \bigcap \left\{P(\gcrd_{R_\phi} A) + \mathcal P_T(\ker \phi) \mid \phi\in C\text{ primitive}\right\} = \{P^C\},
\end{equation}
which implies $P^C = P^{-C}$.

\begin{lemma} \label{lemma:nd-cones}
    The polytope $P^C \in \mathcal P_T(H)$ represented by $P^C_0$ does not depend on $C$.
\end{lemma}
\begin{proof}
    Since the $1$-skeleton of the dual unit ball of $v_A$ is connected, and edges correspond to pairs of \emph{adjacent} open cones (whose closures intersect in a codimension-$1$ boundary region), it suffices to prove $P^C = P^{C'}$ for $C, C'$ adjacent open cones.
    
    Let $\psi$ be a primitive character in the relative interior of $\overline C \cap \overline {C'}$. Then $\psi$ has a small cone-shaped open neighborhood $N$ contained in $\overline C \cup \overline {C'}$, whose integer characters are projectively dense. Applying Corollary~\ref{cor:proj+face} to $P^C, P^\psi$ with $S = N \cap C\cap \Hom(H, \ZZ)$, we obtain $P^C - P^\psi \in \mathcal P_T(\ker \psi)$. Analogously, $P^{C'} - P^\psi \in \mathcal P_T(\ker \psi)$.

    Hence, $P^{C'} - P^C \in \mathcal P_T(\ker \psi)$. The intersection of all kernels $\ker \psi$ for $\psi \in \operatorname{relint}(\overline C \cap \overline {C'})$ is a rank-one subgroup $K(C, C') < H$, parallel to the corresponding edge on the dual unit ball of $v_A$. By Proposition~\ref{prop:inters}, we have
    \begin{equation}
        P^{C'} - P^C \in \mathcal P_T(K(C, C')).
    \end{equation}
    For real characters $\psi \in \operatorname{relint}(\overline C \cap \overline {C'})$ and $\phi \in C$ close to $\psi$ in general position, the semicircular arc $\gamma$ joining $\phi$ to $-\phi$ through $\psi$ in $\Hom(H, \RR)$ intersects a sequence of open cones $C_1 = C, C_2 = C', \dots, C_{z+1} = -C$ in order and is transverse to all cone boundaries. Moreover, $\gamma$ crosses no transverse hyperplane twice. Thus, the groups $K_i = K(C_i, C_{i+1})$, dual to the hyperplanes separating $C_i$ and $C_{i+1}$, are pairwise distinct. Now take a character $\alpha\colon H \to \ZZ$ vanishing on $K_1$ but not on $K_i$, $i > 1$; by arguing as in the proof of Lemma~\ref{lemma:2d-cones}, we obtain $P^C = P^{C'}$.
\end{proof}

This implies the existence of a polytope $P(A) \in \mathcal P_T(H)$, such that for every primitive character $\phi$ in an open cone of $v_A$, $P(A) \in f_A(\phi)$. For all other primitive characters $\psi$, we can argue as in the proof of Lemma~\ref{lemma:nd-cones} to show $P(A) \in P^\psi + \mathcal P_T(\ker \psi) = f_A(\psi)$. This also implies $v_A(\alpha) = \alpha[P(A)]$ for all $\alpha\colon H \to \ZZ$.

Finally, by~\cite[Lemma~4.7]{funke-am}, $P(A)$ is actually a single polytope in $\mathfrak P_T(H)$, concluding the proof of Theorem~\ref{thm:dual-poly}.

\begin{remark}
    This result can be seen as a generalization of~\cite[Theorem~3.14]{kielak-bns} to
rectangular matrices.
\end{remark}
\begin{remark}
    By Proposition~\ref{prop:inters}, the polytope $P(A)$ is determined by finitely many GCRDs of $A$ over $R_\phi$, as $\phi$ ranges over any set of primitive characters generating $\Hom(H, \RR)$.
\end{remark}

\subsection{The Newton polytope of a module}

Given a finitely generated $R$-module $N$ such that $\mathcal D \otimes_R N = 0$, the authors of~\cite{jzksp} define a notion of a $\mathcal P_T(H)$-valued \emph{Newton polytope} $P(N)$. Problem~3 of~\cite{jzksp} inquires whether this invariant is a single polytope in $\mathfrak P_T(H)$; an affirmative answer is implied by the following result.

\begin{thm}\label{thm:jzksp-polytope-equality}
    Let $K$, $H$, $R$, $\mathcal D$ be as in Theorem~\ref{thm:dual-poly}, and let $N$ be a finitely generated $R$-module such that $\mathcal D \otimes_R N = 0$, with presentation matrix $A \in \mathrm{M}(m, n, R)$. Then $P(N) = P(A)$ in $\mathcal P_T(H)$, and hence $P(N)  \in \mathfrak P_T(H)$.
\end{thm}
\begin{proof}
We follow the setup in~\cite[Section~4.2]{jzksp}. For a ring $S$, we define $K_1(S) \coloneqq \mathrm{GL}(S)^{\mathrm{ab}}$ as the abelianization of the infinite general linear group on $S$; if $S$ is a skew field, then $K_1(S) \simeq S^\times / [S^\times, S^\times]$ via the Dieudonné determinant. Moreover, for an abelian category $\mathcal C$, we may construct the group $K_0(\mathcal C)$ as the quotient of the free abelian group on $\mathcal C$ by certain relations corresponding to short exact sequences. This construction applies, for instance, to the category $\mathcal M_S$ of finitely generated $S$-modules, its subcategory $\mathcal C_S$ of finitely generated projective $S$-modules (in which case we simply write $K_0(S)$), or, if $S$ is an Ore localization of $R$, to the category $\mathcal T_S$ of all finitely generated $S$-modules $M$ such that $\mathcal D \otimes_S M = 0$. 

Fix a primitive character $\phi$. We can now apply~\cite[Theorem~5]{quillen} as in~\cite[Section~4.2]{jzksp} to the cases $S = R, R_\phi$, obtaining two Quillen localization long exact sequences 
given by the two rows of the following diagram:
\begin{equation}
\begin{tikzcd}[column sep=scriptsize]
\cdots \arrow[r] & K_1(R) \arrow[r] \arrow[d] & K_1(\mathcal D) \arrow[r, two heads, "p"] \arrow[d, "\operatorname{id}"'] & K_0(\mathcal T_R) \arrow[r, "0"] \arrow[d, "j_\phi"'] & K_0(R) \arrow[r, "\sim"] \arrow[d] & K_0(\mathcal D) \arrow[r] \arrow[d, "\operatorname{id}"'] & 0 \\
\cdots \arrow[r] & K_1(R_\phi) \arrow[r]      & K_1(\mathcal D) \arrow[r, two heads, "p_\phi"]                                 & K_0(\mathcal T_{R_\phi}) \arrow[r, "0"]    & K_0(R_\phi) \arrow[r, "\sim"]      & K_0(\mathcal D) \arrow[r]                                 & 0
\end{tikzcd}
\end{equation}
The vertical arrows are induced by the exact localization functor $R_\phi \otimes_R {-}$, and the diagram commutes by functoriality (as stated in~\cite[Theorem~5]{quillen}). Following~\cite{jzksp}, the maps into $K_0(\mathcal D)$ are isomorphisms, the maps $p$, $p_\phi$ are surjective, and the polytope $P(N)$ is defined as the Newton polytope of any $p$-lift of $[N] \in K_0(\mathcal T_R)$; denote by $d$ one such lift.

Let $D \in \mathrm{M}(n, R)$ be a GCRD of $A$ over $R_\phi$. Then
\begin{equation}
j_\phi(p(d)) = j_\phi([N]) = [\coker (R_\phi \otimes_R A)] = [\coker (R_\phi \otimes_R D)].    
\end{equation}
On the other hand, by \cite[Chapter~III, Corollary~3.1.1 and Exercise~3.6]{k-book}, we have $p_\phi(\det D) = [\coker (R_\phi \otimes_R D)]$, and hence $d$ and $\det D$ are equal up to an element in the image of $K_1(R_\phi)$. Passing to Newton polytopes in $\mathcal P_T(H)$, we can apply Theorem~\ref{thm:dual-poly} to obtain, for some $E \in \mathcal P_T(\ker \phi)$,
\begin{equation}\label{eq:jzksp-poly-eq}
P(N)-P(A) = P(d)-P(D)+E \in P(K_1(R_\phi)) + P_T(\ker \phi).    
\end{equation}
By~\cite[Lemma~3.3~(2)]{friedl-det-monomial}, the Dieudonné determinant of a matrix in $\mathrm{GL}(R_\phi)$ is represented by a monomial in $R_\phi$; hence, its Newton polytope is $\phi$-flat. In other words, we have $P(K_1(R_\phi))\subseteq \mathcal P_T(\ker \phi)$ and $P(N)-P(A) \in P_T(\ker \phi)$. Since $\phi$ was arbitrary, by Proposition~\ref{prop:inters}, we have $P(N) = P(A)$.
\end{proof}
\begin{remark}
    This argument also directly proves $P(N) \in P(D) + \mathcal P_T(\ker \phi) = f_A(\phi)$. Therefore, by~\cite[Lemma~4.7]{funke-am}, $P(N)$ is a single polytope, providing a shorter, if less elementary, proof of Theorem~\ref{thm:dual-poly}. 
\end{remark}

\section{Geometric bounds and arithmetic manifolds}\label{sec:twisted-singer}
The purpose of this section is to estimate the twisted $L^2$-Betti numbers in terms of the ordinary $L^2$-Betti numbers of appropriate hypersurfaces, and to apply the bounds thus obtained, together with the results of the previous sections, to prove our main result (Theorem~\ref{thm:main-final}) on a class of arithmetic manifolds.

First, we show that Conjecture~\ref{conj:main} holds for closed arithmetic manifolds of simplest type in odd dimension, independently of the Atiyah conjecture:

\begin{thm}\label{thm:arithm-seminorm}
    Let $M$ be a closed, odd-dimensional, arithmetic hyperbolic manifold of simplest type. Then the twisted $L^2$-Betti numbers $b_i^{(2)}(\widetilde M; -)$ are polyhedral seminorms for all $i \ge 0$.
\end{thm}
\begin{proof}
    The group $\pi_1(M)$ is virtually special~\cite{arith-virt-spec} and cocompact, and hence satisfies the Atiyah conjecture by~\cite[Theorem~1.2]{virt-spec}. Moreover, $M$ is odd-dimensional hyperbolic, and hence $L^2$-acyclic by~\cite{hyp-l2-acyclic}. Thus, the statement follows by applying Corollary~\ref{cor:main}.
\end{proof}

Let us now prove an estimate for the twisted $L^2$-Betti numbers:
\begin{prop}\label{prop:mayer-vietoris}
    Let $M$ be a closed orientable manifold such that $\pi_1(M)$ satisfies the Atiyah conjecture, and let $F$ be a connected, $\pi_1$-injective embedded hypersurface dual to a primitive cohomology class $\phi \in H^1(M; \mathbb Z)$. If $M$ is $L^2$-acyclic, then the twisted $L^2$-Betti numbers of $M$ are well defined and we have
    \begin{equation}
        b_i^{(2)}(\widetilde M; \phi) \le b_i^{(2)}(\widetilde F).
    \end{equation}
\end{prop}

\begin{proof}
Let $K \coloneqq \ker \phi$, $S \coloneqq \pi_1(F)$, and let $N$ be the infinite cyclic covering of $M$ associated to $\ker \phi$.
Since $F$ lifts to $N$, $S$ is a subgroup of $K$, and the Atiyah conjecture holds for both groups. Assume that $M$ has a regular CW complex structure compatible with that of $F$.

Now we have:
\begin{align*}
    b_i^{(2)}(\til M; \phi) = b_i^{(2)}(\til N) &= \dim_{\mathcal D(K)} H_i(\mathcal D(K) \otimes_{\ZZ K} C_\blt(\til N));
\\  b_i^{(2)}(\til F) &= \dim_{\mathcal D(S)} H_i(\mathcal D(S) \otimes_{\ZZ S} C_\blt(\til F)).
\end{align*}
Since $\mathcal D(K)$ is a flat right $\mathcal D(S)$-module, we can rewrite this as:
\begin{align*}
    b_i^{(2)}(\til M; \phi) &=
    \dim_{\mathcal D(K)} H_i(\mathcal D(K) \otimes_{\ZZ K} C_\blt(\til N));
\\  b_i^{(2)}(\til F)  &=
    \dim_{\mathcal D(K)} H_i(\mathcal D(K) \otimes_{\ZZ S} C_\blt(\til F)).
\end{align*}

Of course, these are simply homology modules with local coefficients in $\mathcal D(K)$, seen first as a $\ZZ K$-module and then as a $\ZZ S$-module. Now construct $N$ in the usual way by gluing infinitely many copies of $M$ cut along $F$. A copy of $F$ splits $N$ into two halves $N_L, N_R$, both having one end. If we slightly enlarge them, we can take $\{N_L, N_R\}$ to be an open cover of $N$, with intersection homeomorphic to $F \times (0,1)$.

Define $K_L\coloneqq \pi_1(N_L), K_R\coloneqq \pi_1(N_R)$. Since $S$ injects in $K$, it must inject in both $K_L$ and $K_R$, which gives $K \simeq K_L *_S K_R$ an amalgamated product structure. The usual arguments from normal forms show that $K_L$ and $K_R$ in turn inject into $K$. All these injections induce inclusions between Linnell skew fields.

By Mayer--Vietoris with local coefficients:
\begin{align*}
\dots &\longrightarrow H_{i+1}(\mathcal D(K) \otimes_{\ZZ K} C_\blt(\til N))
\xrightarrow{\delta} H_{i}(\mathcal D(K) \otimes_{\ZZ S} C_\blt(\til F)) \\ & \longrightarrow H_{i}(\mathcal D(K) \otimes_{\ZZ K_L} C_\blt(\til N_L)) \oplus
H_{i}(\mathcal D(K) \otimes_{\ZZ K_R} C_\blt(\til N_R))
\\ &\longrightarrow H_{i}(\mathcal D(K) \otimes_{\ZZ K} C_\blt(\til N))
\xrightarrow{\delta} \dots
\end{align*}
where again the $S$-, $K_L$-, $K_R$-module structures on $\mathcal D(K)$ are induced by the natural inclusions of $\mathbb Z S$, $\mathbb ZK_L$, $\mathbb ZK_R$ into $\mathbb ZK$ and then into $\mathcal D(K)$.

Note that the map $H_{i}(\mathcal D(K) \otimes_{\ZZ K_L} C_\blt(\til N_L))
\to H_{i}(\mathcal D(K) \otimes_{\ZZ K} C_\blt(\til N))$ is surjective.
In fact, by $\phi$-$L^2$-finiteness of $M$, all $L^2$-Betti numbers of $N$ are finite, so the homology module $H_{i}(\mathcal D(K) \otimes_{\ZZ K} C_\blt(\til N))$ is generated by finitely many cycles over $\mathcal D(K)$. By applying the deck automorphism of $N$ enough times, we can send all these cycles into $N_L$. 

These are also valid cycles in $\mathcal D(K) \otimes_{\ZZ K_L} C_\blt(\til N_L)$, as the latter is a subcomplex of $\mathcal D(K) \otimes_{\ZZ K} C_\blt(\til N)$ with compatible boundary maps, so we get surjectivity. 

Of course, an analogous argument can be made for $N_R$. An immediate consequence is that all boundary maps $\delta$ are zero, so the Mayer--Vietoris long exact sequence decomposes into many short exact sequences.
In particular, by taking dimensions over $\mathcal D(K)$, we get:
\begin{align*}
     b_i^{(2)}(\til M; \phi) + b_i^{(2)}(\til F) &= \dim_{\mathcal D(K)} [H_{i}(\mathcal D(K) \otimes_{\ZZ K_L} C_\blt(\til N_L)) \oplus
H_{i}(\mathcal D(K) \otimes_{\ZZ K_R} C_\blt(\til N_R))] \\ &\ge 2b_i^{(2)}(\til M; \phi),
\end{align*}
from which the result follows immediately.
\end{proof}

Due to Agol's virtual fibration theorem, such $\pi_1$-injective hypersurfaces are ubiquitous in $3$-manifolds, but finding them in higher-dimensional hyperbolic manifolds appears to be much harder. In this direction, we introduce a result of Bergeron--Millson--Moeglin, which deals with arithmetic manifolds of simplest type corresponding to congruence subgroups of an arithmetic lattice $\mathrm O(f, \mathcal O_k)$ (also called \emph{congruence manifolds}):
\begin{prop}[{see~\cite[Theorem~1.2]{virt-geod}}]\label{prop:bmm}
    Let $M$ be a closed, orientable, congruence arithmetic hyperbolic manifold of simplest type, with $\dim M > 3$. Then $H^1(M; \mathbb Q)$ is spanned by the Poincaré duals of classes of immersed totally geodesic hypersurfaces.
\end{prop}

We shall now prove the following:
\begin{thm}\label{thm:twisted-singer}
    Let $M$ be a closed congruence arithmetic hyperbolic $(2n+1)$-manifold of simplest type. Then the twisted $L^2$-Betti numbers $b_i^{(2)}(\widetilde M; -)$ vanish identically for $i \ne n$.
\end{thm}
\begin{proof}
    The statement holds for $n \le 1$ as shown in~\cite[Section~7.5.1]{twi-l2}; hence, assume $n \ge 2$. As in the proof of Theorem~\ref{thm:arithm-seminorm}, $\pi_1(M)$ satisfies the Atiyah conjecture. Moreover, up to a double cover, we may assume $M$ orientable.
    Let $F$ be an immersed, oriented, totally geodesic hypersurface in $M$. 
    The image of $\pi_1(F)$ is separable in $\pi_1(M)$ by~\cite[Lemma~1.7]{arith-virt-spec}; then, by a standard argument, we can find a finite normal cover $p\colon N \twoheadrightarrow M$ where $F$ lifts to finitely many connected, embedded, oriented, totally geodesic hypersurfaces $F_1, \dots, F_k$. We then have the pullback formula $p^*[F] = [F_1] + \dots + [F_k]$, where each summand is primitive or zero. Since a totally geodesic submanifold is $\pi_1$-injective, we have $b_i^{(2)}(\til N; [F_j]) \le b_i^{(2)}(\til F_j)$, by Proposition~\ref{prop:mayer-vietoris} (or trivially if $[F_j] = 0$).

    Since any $F_j$ is closed hyperbolic, it satisfies the Singer conjecture. Therefore, its $L^2$-Betti numbers vanish in dimension not equal to $n$, and so do the twisted $L^2$-Betti numbers $b_i^{(2)}(\til N; [F_j])$. By the triangle inequality and multiplicativity along finite covers, we have
    \begin{equation}
        (\deg p)\cdot b_i^{(2)}(\til M; [F]) = b_i^{(2)}(\til N; [F_1] + \dots + [F_k])
        \le \sum_{j = 1}^k b_i^{(2)}(\til N; [F_j]) = 0
    \end{equation}
    for $i \ne n$. By Proposition~\ref{prop:bmm}, the admissible choices for $F$ include a basis of $H^1(M; \mathbb Q)$ up to Poincaré duality, and hence the seminorm $b_i^{(2)}(\widetilde M; -)$ is identically zero.
\end{proof}

By combining Theorems~\ref{thm:arithm-seminorm} and~\ref{thm:twisted-singer}, 
we obtain our main theorem:

\begin{namedthm*}{Theorem~\ref{thm:main-final}}
    Let $M$ be a closed congruence arithmetic hyperbolic $(2n+1)$-manifold of simplest type. Then $(-1)^n \cdot \chi^{(2)}(\widetilde M; -) = b_n^{(2)}(\widetilde M; -)$ is a seminorm, and the other twisted $L^2$-Betti numbers $b_i^{(2)}(\widetilde M; -)$, with $i\ne n$, are identically zero.
\end{namedthm*}

\begin{remark}
    The requirement that the manifolds in Proposition~\ref{prop:bmm} be specifically \emph{congruence} arithmetic manifolds stems from the number-theoretic nature of its proof, and might not be fundamentally necessary.
\end{remark}

\begin{remark}
    Given the hypotheses of Theorem~\ref{thm:main-final}, define $\Sigma(M)$ to be the set of closed, oriented, immersed totally geodesic hypersurfaces of $M$. It is natural to define a seminorm
    \begin{equation}
        s(\widetilde M; \phi) = \inf \left\{\sum_{i = 1}^k |\alpha_i|\cdot |\chi(S_i)| \biggm| \phi = \sum_{i=1}^k \alpha_i [S_i], S_i \in \Sigma(M)\right\},
    \end{equation}
    where we allow the value $+\infty$ if $\dim M = 3$. By taking normal covers as in the proof of Theorem~\ref{thm:twisted-singer}, we can apply Proposition~\ref{prop:mayer-vietoris} to obtain the inequality $b_n^{(2)}(\widetilde M; -) \le s(\widetilde M; -)$. It is natural to ask whether there is any further relation between the two invariants, and under which conditions they are equal. A possible experimental approach is to compute middle-dimensional $L^2$-Betti numbers of geodesic hypersurfaces, and hence the seminorm $s(\widetilde M; -)$, via number-theoretic formulas, and to compare with the results of the algorithm in~\cite{twi-l2} for the twisted $L^2$-Euler characteristic.
\end{remark}

In the above discussion, the significance of totally geodesic hypersurfaces is that they are $\pi_1$-injective and aspherical (in other words, they are injective in all homotopy groups), and they satisfy the Singer conjecture. More generally, there are at least two classes of immersed hypersurfaces that are $\pi_1$-injective and aspherical: virtual fibers of fibrations, and hypersurfaces with principal curvatures in $[-1, 1]$ (the latter following from~\cite[Theorem~3.6]{pi1-inj-asph}). If such hypersurfaces span the first cohomology of a closed aspherical $(2n+1)$-manifold $M$, and lift to embeddings in finite covers, then Theorem~\ref{thm:main-final} will also hold for $M$, conditional on the Atiyah and Singer conjectures.

As mentioned before, in dimension $>3$, realizing a cohomology class by either a $\pi_1$-injective hypersurface or an aspherical one is a highly nontrivial problem, which does not seem to appear in the literature. In some sense, we expect an abundance of hypersurfaces of small curvature: see e.g. Kahn--Markovic~\cite{kahn-markovic} for $3$-manifolds, and Han--Jiang~\cite{han-jiang} in dimension $\ge 3$ for arithmetic manifolds of simplest type. However, these results do not obviously give any control on the cohomology classes these hypersurfaces realize. In another direction, there is also reason to believe that some form of virtual fibration holds in high dimensions: see a recent survey by Kielak~\cite{kielak-fiber}, which outlines possible proof strategies. This problem appears to be very difficult, and is deeply connected with $L^2$-invariants. It would be interesting to pursue further research along these lines.

\clearpage

\printbibliography

@article{funke-am,
    author = {F. Funke},
    title = {The $L^2$-torsion polytope of amenable groups},
    number = {Doc. Math.},
    issue = {23},
    date = {2018},
    pages = {1969--1993},
}

@article{l2thur,
  title = {$L^2$‐Euler characteristics and the Thurston norm},
  volume = {118},
  DOI = {10.1112/plms.12202},
  number = {4},
  journal = {Proc. Lond. Math. Soc.},
  publisher = {Wiley},
  author = {Friedl, S. and L\"{u}ck, W.},
  date = {2019-04},
  pages = {857--900}
}

@article{twi-l2,
    author = {J. G. Chen},
    title = {Computing the twisted $L^2$-Euler characteristic},
    journal = {Groups Geom. Dyn.},
    date = {2025},
    note = {Published online first}
}

@article{kielak-bns,
    author = {D. Kielak},
    title = {The Bieri--Neumann--Strebel invariants via Newton polytopes},
    journal = {Invent. Math.},
    date = {2020},
    number = {219},
    pages = {1009--1068},
}

@misc{kapovich,
author = {M. Kapovich},
title = {Integer norms are polyhedral},
url = {https://www.math.ucdavis.edu/~kapovich/EPR/norms.pdf}
}

@article{thurston,
  title={A norm for the homology of 3-manifolds},
  author={W. P. Thurston},
  journal={Mem. Am. Math. Soc.},
  year={1986},
  volume={59},
  pages={99--130},
}

@article{virt-geod,
  title = {Hodge Type Theorems for Arithmetic Manifolds Associated to Orthogonal Groups},
  DOI = {10.1093/imrn/rnw067},
  journal = {Int. Math. Res. Not.},
  author = {Bergeron, N. and Millson, J. and Moeglin, C.},
  date = {2017-08},
  pages = {4495--4624}
}

@inproceedings{seminorm-conj,
    author = {S. Friedl and W. Lück and S. Tillmann},
    title = {Groups and polytopes},
    booktitle = {Breadth in contemporary topology},
    volume = {102},
    pages = {57--77},
    year = {2019},
    publisher = {Amer. Math. Soc.},
    series = {Proc. Sympos. Pure Math.}
}

@misc{jzksp,
      title={Thurston norm, polytopes and splitting complexity}, 
      author={A. Jaikin-Zapirain and M. Kudlinska and P. Sánchez-Peralta},
      year={2026},
      eprint={2606.31774},
      archivePrefix={arXiv},
      primaryClass={math.GR},
      %url={https://arxiv.org/abs/2606.31774}, 
}

@article{kahn-markovic,
  title = {Immersing almost geodesic surfaces in a closed hyperbolic three manifold},
  volume = {175},
  ISSN = {0003-486X},
  DOI = {10.4007/annals.2012.175.3.4},
  number = {3},
  journal = {Ann. Math.},
  author = {Kahn,  J. and Markovic,  V.},
  date = {2012-05},
  pages = {1127–-1190}
}

@misc{han-jiang,
      title={Asymptotically geodesic hypersurfaces and the fundamental groups of hyperbolic manifolds}, 
      author={X. H. Han and R. Jiang},
      year={2026},
      eprint={2603.24869},
      archivePrefix={arXiv},
      primaryClass={math.GT},
     % url={https://arxiv.org/abs/2603.24869}, 
}

@misc{kielak-fiber,
      title={Virtual fibring of manifolds and groups}, 
      author={D. Kielak},
      year={2025},
      eprint={2510.01805},
      archivePrefix={arXiv},
      primaryClass={math.GR},
}

@book{l2,
    author={W. Lück},
    title={$L^2$-Invariants: Theory and Applications to Geometry and K-Theory},
    year={2002},
    publisher={Springer},
    %series = {Ergeb. Math. Grenzgeb., 3. Folge},
    %number = {44}
}

@book{cohn,
    author = {P. M. Cohn.},
    title = {Skew fields. Theory of general division rings},
    publisher = {Cambridge University Press},
    year = {1995},
    series = {Encyclopedia of Mathematics and its Applications},
    number = {57},
}

@article{funke,
  title = {The integral polytope group},
  volume = {21},
  ISSN = {1615-715X},
  DOI = {10.1515/advgeom-2019-0029},
  number = {1},
  journal = {Adv. Geom.},
  publisher = {Walter de Gruyter GmbH},
  author = {Funke,  F.},
  date = {2021},
  pages = {45--62}
}

@book{ore-flat,
    author = {B. Stenström},
    title = {Rings of Quotients},
    publisher = {Springer-Verlag},
    year = {1975} 
}

@article{virt-spec,
  title = {The strong Atiyah conjecture for virtually cocompact special groups},
  volume = {359},
  DOI = {10.1007/s00208-014-1007-9},
  number = {3-4},
  journal = {Math. Ann.},
  publisher = {Springer Science and Business Media LLC},
  author = {Schreve, K.},
  date = {2014-02},
  pages = {629--636}
}

@article{arith-virt-spec,
  title = {Hyperplane sections in arithmetic hyperbolic manifolds},
  volume = {83},
  DOI = {10.1112/jlms/jdq082},
  number = {2},
  journal = {J. Lond. Math. Soc.},
  publisher = {Wiley},
  author = {Bergeron, N. and Haglund, F. and Wise, D. T.},
  date = {2011-02},
  pages = {431--448}
}

@book{k-book,
  title = {The $K$-book},
  DOI = {10.1090/gsm/145},
  journal = {Graduate Studies in Mathematics},
  publisher = {American Mathematical Society},
  author = {Weibel, C.},
  date = {2013-06}
}

@inproceedings{quillen,
  author    = {Quillen, D.},
  title     = {Higher algebraic {$K$}-theory. {I}},
  booktitle = {Algebraic {$K$}-theory, {I}: Higher {$K$}-theories},
  series    = {Lecture Notes in Mathematics},
  volume    = {341},
  pages     = {85--147},
  publisher = {Springer},
%  location  = {Berlin-New York},
  date      = {1973},
  venue     = {Battelle Memorial Institute, Seattle, Washington},
  eventdate = {1972},
}

@article{friedl-det-monomial,
  author       = {Friedl, S.},
  title        = {Reidemeister torsion, the {Thurston} norm and {Harvey}'s invariants},
  journaltitle = {Pac. J. Math.},
  volume       = {230},
  number       = {2},
  date         = {2007},
  pages        = {271--296},
  doi          = {10.2140/pjm.2007.230.271},
}

@article{gcrd-bcl,
  title = {Fraction-free row reduction of matrices of Ore polynomials},
  volume = {41},
  DOI = {10.1016/j.jsc.2005.10.002},
  number = {5},
  journal = {J. Symb. Comput.},
  publisher = {Elsevier BV},
  author = {Beckermann, B. and Cheng, H. and Labahn, G.},
  date = {2006-05},
  pages = {513--543}
}

@article{gcrd-bl,
  title = {Fraction-Free Computation of Matrix Rational Interpolants and Matrix GCDs},
  volume = {22},
  DOI = {10.1137/s0895479897326912},
  number = {1},
  journal = {SIAM J. Matrix Anal. Appl.},
  publisher = {Society for Industrial & Applied Mathematics (SIAM)},
  author = {Beckermann, B. and Labahn, G.},
  date = {2000-01},
  pages = {114--144}
}

@article{gcrd-nv,
  title = {Revisiting the Matrix Polynomial Greatest Common Divisor},
  volume = {44},
  DOI = {10.1137/22m1531993},
  number = {3},
  journal = {SIAM J. Matrix Anal. Appl.},
  publisher = {Society for Industrial & Applied Mathematics (SIAM)},
  author = {Noferini, V. and Van Dooren, P.},
  date = {2023-08},
  pages = {1164--1188}
}

@article{pi1-inj-asph,
  title = {Uniqueness and non‐uniqueness for the asymptotic Plateau problem in hyperbolic space},
  volume = {132},
  DOI = {10.1112/plms.70121},
  number = {1},
  journal = {Proc. Lond. Math. Soc.},
  publisher = {Wiley},
  author = {Huang, Z. and Lowe, B. and Seppi,  A.},
  year = {2026}
}

@article{hyp-l2-acyclic,
    author = {J. Dodziuk},
    journal = {Proc. Am. Math. Soc.},
    number = {3},
    pages = {395--400},
    publisher = {American Mathematical Society},
    title = {$L^2$ harmonic forms on rotationally symmetric Riemannian manifolds},
    volume = {77},
    year = {1979}
}

@book{convex-bodies,
  title = {Convex Bodies: The Brunn--Minkowski Theory},
  edition = {Second expanded edition},
  DOI = {10.1017/cbo9781139003858},
  publisher = {Cambridge University Press},
  author = {Schneider, R.},
  date = {2013-12}
}

\end{document}